\documentclass{amsart}
\usepackage{amsmath}
\usepackage{amssymb}
\usepackage{graphicx}
\input xy
\xyoption{all}
\newtheorem{theorem}{Theorem}[section]

\newtheorem{lemma}[theorem]{Lemma}
\newtheorem{corollary}[theorem]{Corollary}
\newtheorem{proposition}[theorem]{Proposition}
\theoremstyle{definition}

\newtheorem{remark}[theorem]{Remark}%[section]

\newcommand{\M}{\mbox{\rm M}}
\newcommand{\trdeg}{\mbox{\rm trdeg }}
\newcommand{\End}{\mbox{\rm End}}

\newcommand{\I}{\mbox{\rm I}}

\newcommand{\N}{\mathbb N}
\newcommand{\Z}{\mathbb{Z}}

\newcommand{\Q}{\mathbb{Q}}

\newcommand{\GL}{{\rm GL}}
\newcommand{\Gal}{{\rm Gal}}
\newcommand{\Char}{\mbox{{\rm char}}}

\begin{document}
    \title[On subnormal subgroups containing a solvable subgroup]{On subnormal subgroups  in division rings containing a non-abelian solvable subgroup}
    \author[Bui Xuan Hai]{Bui Xuan Hai}\author[Mai Hoang Bien]{Mai Hoang Bien}
    %\address{University of Science, VNU-HCMC, Vietnam}
   \email{bxhai@hcmus.edu.vn;  mhbien@hcmus.edu.vn}
\address{Faculty of Mathematics and Computer Science, VNUHCM-University of Science,  227 Nguyen Van Cu Str., Dist. 5, Ho Chi Minh City, Vietnam.}

\keywords{subnormal subgroups, free subgroup, solvable subgroup, division ring, quotient division ring. \\
\protect \indent 2010 {\it Mathematics Subject Classification.} 16K20, 20E05, 16S85, 16S36.}

 \maketitle

 \begin{abstract} Let $D$ be a division ring with center $F$ and $N$ a subnormal subgroup of the multiplicative group $D^*$ of $D$.  Assume that $N$ contains a non-abelian solvable subgroup. In this paper, we study the problem on the existence of non-abelian free subgroups in $N$. In particular, we show that if either $N$ is algebraic over  $F$ or $F$ is uncountable, then $N$ contains a non-abelian free subgroup. 
\end{abstract}
\section{Introduction}        
Let $D$ be a division ring with center $F$ and $D^*$ denotes the multiplicative group of $D$. Among the classical results  on solvable subgroups in division rings, we note the earliest theorem due to L. K. Hua \cite{hua} which states that if $D^*$ is solvable, then $D$ is commutative. Since then, there are many works devoted to this subject. In 1960, Huzubazar \cite{huzu} showed that in  division rings there are no non-central locally nilpotent subnormal subgroups. Later, Stuth \cite{stuth} proved that if a subnormal subgroup $G$ of $D^*$ is solvable, then $G$ must be central, that is, $G\subseteq F$. Several theorems of such a kind of other authors one can find, for example in \cite{hai,hai-ha,hai-thin1,hai-thin2,stuth} and in the references therein. This subject has close connection to  the problem on the existence of non-abelian free subgroups in division rings. For instance, in \cite[Theorem 2.2.1]{Pa_HaMaMo_14}, it was proved that in the case when $D$ is centrally finite, that is, $\dim_FD<\infty$, a subgroup $G$ of $D^*$ contains no non-abelian free subgroup if and only if $G$ contains a solvable normal subgroup of finite index. Our paper follows this side. Namely, we establish some relations between the existence of  non-abelian solvable subgroups and non-abelian free subgroups in subnormal subgroups of $D^*$. Note that this subject is not quite strange. Indeed, in \cite[Theorem 2]{Pa_Li_78}, Lichtman showed that if a normal subgroup of $D^*$ contains a non-abelian nilpotent-by-finite subgroup, then it contains a non-abelian free subgroup. These results were extended in \cite{Pa_Go_Pa_15} for a normal subgroup of $D^*$ containing a non-abelian solvable subgroup with some additional conditions.  In this paper, we consider a subnormal subgroup of $D^*$ containing a non-abelian solvable subgroup instead of a normal subgroup containing a non-abelian nilpotent-by-finite subgroup. In fact, assume that $N$ is a subnormal subgroup of $D^*$ containing a non-abelian solvable subgroup. We will prove that if $N$ is algebraic over $F$ or the transcendence degree  of $F$ over its prime subfield $\mathbb{P}$ is infinite, then $N$ contains a non-abelian free subgroup (see theorems~\ref{4.1}, ~\ref{4.2} in the text).

 The ideas of techniques of proofs in this paper are taken from \cite{Pa_Go_Pa_15}, and they can be shortly explained as the following: assume that $N$ is a subnormal subgroup of $D^*$ and $G$ is a non-abelian solvable subgroup of $N$. Then, there exists a division subring $D_1$ of $D$ which is isomorphic to the quotient division ring $K(t,\sigma)$ of a skew polynomial ring $K[t,\sigma]$ in $t$ over some subfield $K$ with respect to an automorphism field $\sigma$ of $K$, and $D_1\cap N$ contains $t$.  In some special cases, we may continue to choose a division subring $D_2$ of $D_1$ such that $D_2$ is isomorphic to the quotient division ring $L(x,y)$ of a polynomial ring $L[x,y]$ in two commutating indeterminates over a suitable division ring $L$. The construction of $D_1$ and $D_2$ will be presented in Section 4 and with the above reason, in Section 2, we consider the existence of non-abelian free subgroups of subnormal subgroups in the quotient division ring of a skew polynomial ring in one indeterminate over a field. In Section 3, our base division ring is the quotient division ring $D(x,y)$ of a polynomial ring $D[x,y]$ in two commuting indeterminates over a division ring $D$. 

Remark that recently, the problem on the existence of free subobjects in the quotient division rings of skew polynomial rings have been received a considerable attention. For instance, the existence of free subalgebra in the quotient division rings of Ore polynomial rings was studied in \cite{Pa_BeGo_16, Pa_BeRo_14, Pa_BeRo_12,Pa_Go_Pa_15} while in \cite{Pa_Go_17}, Gon\c calves dealt with the existence of  non-abelian free subgroups of a normal subgroup of the quotient division ring of a certain skew polynomial ring. Hence, the next section of the present paper is itself an interesting topic.

\section{The existence of free subgroups in $K(x,\sigma)$}

Let $G$ be a group. If $a,b\in G$, then we denote $[{}_1a,b]=aba^{-1}b^{-1}$ and $[{}_{i+1}a,b]=[a, [{}_ia,b]]$ for an integer  $i\geq 1$. The following lemma will be used in the next. 
\begin{lemma}\label{3.1} Let $G$ be a group, and assume that $N$ is a subnormal subgroup of $G$ with a series of subgroups
$$N=N_n\triangleleft N_{n-1}\triangleleft \cdots \triangleleft N_0=G.$$ 
	
If $a\in N$ and $b\in G$, then $[{}_ia,b]\in N_i$ for every positive integer $i$. In particular, $[{}_na,b]\in N$.
\end{lemma}
\begin{proof} The proof is simple by induction on $i$.
\end{proof}
Let $R$ be a ring and $\sigma : R\to R$ an endomorphism of $R$. A \textit{skew polynomial ring} in an indeterminate $x$ over  $R$ corresponding to $\sigma$, denoted by $R[x,\sigma]$, is defined as the following: 

As a set, $R[x,\sigma]$ consists of all polynomials in an indeterminate $x$ with coefficients taken from $R$, that is,  $$R[x,\sigma]=\{a_0+a_1x+\cdots+a_nx^n\mid n\in \N, a_i\in R\text{ for every } 1\le i\le n\}.$$
The addition in $R[x,\sigma]$ is usual while polynomials are multiplied formally with the rule  $xa=\sigma(a)x$. If $R=K$ is a field and $\sigma$ is an automorphism, then  $K[x,\sigma]$ is a PLID (\textit{principal left ideal domain}) and also a PRID (\textit{principal right ideal domain}), that is, every one-sided ideal in $K[x,\sigma]$ is generated by one element. In this case, we denote by $K(x,\sigma)$ the quotient division ring of $K[x,\sigma]$. Every element of  $K(x,\sigma)$ has the form $f(x)g^{-1}(x)$, where $f(x),g(x)\in K[x,\sigma]$.

 In this section, we consider the existence of non-abelian free subgroup in a subnormal subgroup of $K(x,\sigma)$ containing $x$. The recent results on the existence of free subobject in $K(x,\sigma)$ can be found in \cite{Pa_BeGo_16,Pa_BeRo_14,Pa_BeRo_12,Pa_Go_17}.
\begin{lemma}\label{2.1} Let $K$ be a field,  and  $\sigma$ an automorphism of $K$ with fixed subfield $E=\{a\in K\mid \sigma (a)=a \}$. If $a\in K$ and  $E(\sigma^n(a)\mid n\in \N)$ is infinitely generated, then the set 
$$S=\{ \sigma^{i_1} (a)\sigma^{i_2} (a)\cdots \sigma^{i_n} (a)\mid n\ge 1, 0\le i_1<i_2,\cdots<i_n \}$$ 
is linear independent over $E$.
\end{lemma}
\begin{proof} Deny the conclusion, and let  $\sum\limits_{j=1}^\mu \alpha_j \sigma^{i_{1,j}} (a)\sigma^{i_{2,j}} (a)\cdots \sigma^{i_{n,j}}(a)=0$ for some $\mu\in \N$ and $\alpha_j\in E$. Assume that $\mu$ is the smallest integer satisfying such a relation, i.e., this relation is shortest. Put $m=\max\{i_{n,j}\mid 1\le j\le \mu \}$. Then, this dependence relation may be written in the form $$s\sigma^m(a)+t=0, \mbox{ where } s,t\in E(\sigma^i(a)\mid i<m).$$ 
By the choice of $\mu$, we have $s\ne 0$, which implies that 
$$\sigma^m(a)=s^{-1}t\in E(\sigma^i(a)\mid i\le m-1).$$
As a corollary,

$\sigma^{m+1}(a)=\sigma(\sigma^m(a))\in E(\sigma^i(a)\mid i\le m)$

\hspace*{1.44cm}$=  E(\sigma^i(a)\mid i\le m-1)(\sigma^m(a)) \subseteq  E(\sigma^i(a)\mid i\le m-1).$\\
Thus, $E(\sigma^n(a)\mid n\in \N)$ is finitely generated that contradicts to the hypothesis.	
\end{proof}
Let $K$ be a field, $\sigma$ an automorphism of $K$ and $D=K(x,\sigma)$. Assume that $a\in K$ and $N$ is a subnormal subgroup of $D^*$ containing $x$ with the following  series of subgroups  
$$N=N_n\triangleleft N_{n-1}\triangleleft \cdots \triangleleft N_0=D^*.$$ 
Without loss of generality, we can suppose that $n\geq 1$ because if $N=D^*$, then we can set $N_1=N_0=D^*$. For the convenience, we denote $a^{(i)}=\sigma^i(a)$ for $a\in D$. Consider the sets  $\{y_i=a^{(i)}x^{2n}\mid 0\le i< n \}$ and  $\{A_i=1+y_i\mid 0\le i<n \}$. 
\begin{lemma}\label{2.2} With the assumption and the symbols as above, we have 
	
		(i) The element $[{}_nx,A_0]$ belongs to the intersection between $N$ and the subgroup $\langle A_i\mid 0\le i<n\rangle$ of $D^*$ generated by all $A_i, 0\le i<n$. Also, its first and last factors are  $A_n$ and $A_{n-1}^{-1}$ respectively.
		
		(ii) The element $[{}_nx,A_0^{-1}]$ belongs to the intersection between $N$ and the subgroup $\langle A_i\mid 0\le i<n\rangle$ of $D^*$ generated by all $A_i, 0\le i<n$. Also, its first and last factors are  $A_n^{-1}$ and $A_{n-1}$ respectively.	
\end{lemma}
\begin{proof} We shall prove the Part (i). The proof of the Part (ii) is similar. In view of Lemma ~\ref{3.1}, $[{}_ix,A_0] \in N_i$ for all $i$; in particular, $[{}_nx,A_0] \in N_n$. Now, we prove by induction on $i>0$ the following properties:
	
	(a) $[{}_ix,A_0]\in \langle A_1,\cdots,A_i\rangle$; and
	
	(b) $[{}_ix,A_0]$ has $A_i$ and $A_{i-1}^{-1}$ as its initial and last factor respectively.
	
	For any $i > 0$, we have			 $$xA_ix^{-1}=x(1+a^{(i)}x^{2n})x^{-1}=1+a^{(i+1)}x^{2n}=A_{i+1},  \mbox{ and }$$
	$$xA_i^{-1}x^{-1}=(xA_ix^{-1})^{-1}=A_{i+1}^{-1}\eqno (1)$$ 
	which implies  
	$$x\langle A_1,A_2,\cdots,A_i\rangle x^{-1}\subseteq \langle A_1,A_2,\cdots,A_{i+1}\rangle.\eqno (2)$$
We see that for $i=1$, (a) and (b) are true in view of the following equalities:	
	$$[{}_1x,A_0]=(xA_0x^{-1})A_0^{-1}=A_1A_0^{-1}.$$
Assume that (a) and (b) hold for $0 < i < n$. We shall prove that they hold for $i+1$. Indeed, by induction assumption and in view of (1), we have
	$$[{}_{i+1}x,A_0]=[x,[{}_ix,A_0]]=(x[{}_ix,A_0]x^{-1})[{}_ix,A_0]^{-1}\in \langle A_1,A_2,\cdots,A_{i+1}\rangle.$$
Also, $[{}_ix,A_0]=A_iBA_{i-1}$.  So, in view of (1), it follows 
	
	$[{}_{i+1}x,A_0]=x(A_iBA_{i-1}^{-1})x^{-1}(A_iBA_{i-1}^{-1})^{-1}$
	
	\hspace*{1.55cm}$=(xA_ix^{-1})xBA_{i-1}^{-1}x^{-1}A_{i-1}B^{-1}A_i^{-1}$
	
	\hspace*{1.55cm}$=A_{i+1}xBA_{i-1}^{-1}x^{-1}A_{i-1}B^{-1}A_i^{-1}.$
Thus, (a) and (b) both hold for $i+1$, and the proof is now complete.
\end{proof}
Let $R$ be a ring and $\sigma$ be an automorphism of $R$. The ring  $$R[[x,\sigma]]=\Big\{\alpha =\sum\limits_{i=n}^\infty a_ix^i\mid n\in \Z, a_i\in R \Big\}$$
consisting of formal Laurent series over $R$ with the usual addition and the multiplication defined formally with the rule $xa=\sigma(a)x$ is called a \textit{skew Laurent series ring}. Clearly, $R[x, \sigma]$ is a subring of $R[[x, \sigma]]$. If  $R=K$ is a field, then  $K[[x,\sigma]]$ is a division ring \cite[Example 1.8]{lam} and  $K(x,\sigma)$ can be considered as a division subring of  $K[[x,\sigma]]$ via obvious injection. In particular, we have the following remark.

\begin{remark}\label{r2.4} If $a\in K^*$, then the element $(1+ax^m)^{-1}\in K(x, \sigma)$ can be written in the form:
	$$(1+ax^m)^{-1}=\sum\limits_{i=0}^{\infty}(-1)^ia^ix^{mi} =1-ax^m+ \sum\limits_{i=2}b_ix^{mi}$$ in $K[[x,\sigma]]$ with $b_i\ne 0$ for all $i\geq 2$.
\end{remark}
\begin{theorem}\label{2.3} Let $K$ be a field, and $\sigma$  an automorphism of $K$ with fixed subfield $E=\{a\in K\mid \sigma (a)=a\}$. 
Assume that $D=K(x,\sigma)$ is the quotient division ring of $K[x,\sigma]$, and at least one of the following conditions holds:	
	
	(i) There exists an element $a\in K$ such that $E(a,\sigma(a),\sigma^2(a),\cdots )$ is infinitely generated.
	
	(ii) There exists an element $a\in K\backslash E$ which is algebraic over $E$.
	
If $N$ is a subnormal subgroup of $D^*$ containing $x$, then $N$ contains a non-abelian free subgroup. 
\end{theorem}
\begin{proof} As in the proof of Lemma \ref{2.2}, for $a\in K$ and $i\in \N$, we denote  $a^{(i)}=\sigma^i(a)$. Consider a series of subgroups
	$$N=N_n\triangleleft N_{n-1}\triangleleft \cdots \triangleleft N_0=D^*.$$ 
	Without loss of generality, we can suppose that $ n\geq 1$ because if it is necessary, we can put $N_1=N_0=D^*$.
	
	(i) Assume that there exists $a\in K$ such that  $E(a^{(i)}\mid i\in \N)$ is infinitely generated. We shall prove that the elements  $[{}_nx,A_0], [{}_n,A_0^{-1}]$ defined as in Lemma \ref{2.2} generate a non-abelian free subgroup in  $N$.
	
	Firstly, we claim that the subset 
	$\{y_i=a^{(i)}x^{2n}\mid 0\le i\le n \}$ generates a free $E$-algebra. Indeed, deny the conclusion, that is, there exists a non-zero polynomial $F(z_0,z_1,\cdots,z_n)$ in the free $E$-algebra in $n+1$ indeterminates $z_0,z_1,\cdots, z_n$ such that $F(y_0,y_1,\cdots, y_n)=0.$ We assume that $$F(z_0,z_1,\cdots,z_n)= F_1(z_0,z_1,\cdots,z_n)+F_2(z_0,z_1,\cdots,z_n)+\cdots+ F_\mu(z_0,z_1,\cdots, z_n),$$ where $F_\ell(z_0,z_1,\cdots,z_n)=\sum\limits_{j_1+j_2+\cdots+j_\nu=\ell} \alpha_{j_1,j_2,\cdots,j_\nu} z_{\ell_1}^{j_1}z_{\ell_2}^{j_2}\cdots z_{\ell_\nu}^{j_\nu}$ is a homogeneous polynomial of degree $\ell$ with coefficients in $E$  for every $1\le \ell\le \mu$. Let $s$ be a positive integer such that $F_s(z_0,z_1,\cdots,z_n)$ is non-zero (the existence of $s$ is obvious since $F$ is non-zero). 
	Observe that the degree of $F_\ell (y_0,y_1,\cdots,y_n)$ in the indeterminate $x$ is $2n\ell$, so $F(y_0,y_1,\cdots, y_n)=0$ if and only if $F_\ell(y_0,y_1,\cdots, y_n)=0$ for every $1\le \ell\le \mu$. In particular, $$0=F_s(y_0,y_1,\cdots,y_n)=\sum\limits_{j_1+j_2+\cdots+j_\nu=s} \alpha_{j_1,j_2,\cdots,j_\nu} y_{\ell_1}^{j_1}y_{\ell_2}^{j_2}\cdots y_{\ell_\nu}^{j_\nu}.$$ After moving all coefficients to the left and the indeterminate $x$ to the right, one has  
	$$F_s(y_0,y_1,\cdots,y_n)=\sum\limits_{j_1+j_2+\cdots+j_\nu=s} \alpha_{j_1,j_2,\cdots,j_\nu}b_{j_1,j_2,\cdots,j_\nu}x^{2ns},$$
	where 
	$b_{j_1,j_2,\cdots,j_\nu}=a^{(j_1)} a^{(j_1+2n)}\cdots a^{(j_1+2(j_1-1)n)} a^{(j_2+2j_1n)}\cdots$
	
	\hspace*{2.5cm}$\cdots a^{(j_2+2j_1n+2n)}\cdots a^{(j_\nu+2(j_\nu-1)n+2j_{\nu-1}n+\cdots 2j_1n)}$.\\
This implies that  $\sum\limits_{j_1+j_2+\cdots+j_\nu=s}\alpha_{j_1,j_2,\cdots,j_\nu} b_{j_1,j_2,\cdots,j_\nu}=0$. Since every coefficient $b_{j_1,j_2,\cdots,j_\nu}$ belongs to $S$, by Lemma~\ref{2.1}, $\alpha_{j_1,j_2,\cdots,j_\nu}=0$ which contradicts to the assumption that  $F_s(z_0,z_1,\cdots,z_n)$ is non-zero.

	Next, we show that the set $\{A_i=1+y_i\mid 0\le i\le n \}$ generates a free subgroup in $(K(x,\sigma))^*$. Assume that there is a relation $$A_{i_1}^{n_1}A_{i_2}^{n_2}\cdots A_{i_t}^{n_t}=1,$$ where $0\le i_j<n$ and all $n_j\ne 0$. We seek a contradiction by describing the representation of this relation in  $K[[x,\sigma]]$. First of all, we need to know the representation of $A_i^j$  in $K[[x,\sigma]]$ with $j\in \Z^*$.  If $j> 0$, then $A_{i}^{j}$ is a polynomial in $K[x,\sigma]\subseteq K[[x,\sigma]]$, so we can write $A_{i}^{j}=1+y_i+A_i'$ in which either $A_i'=0$ (in case $j=1$) or the degree of $A_i'$ is greater than $2n$ (in the indeterminate $x$). If $j<0$, then by Remark~\ref{r2.4}, one has $A_i^{-1}=1-y_i + y_i^2+\cdots\in K[[x,\sigma]]$. Hence, if $\Char(D)=0$, then $A_i^j=1-y_i+A_i'$ where $A_i'$ is a series in $K[[x,\sigma]]$ whose degree of the first term is greater than $2n$. If $\Char(K)=p>0$, then we write $j=p^\lambda j_1$ with $(p,j_1)=1$. Observe that
	$$A_i^j=(1-y_i+y_i^2+\cdots)^j=((1-y_i+y_i^2+\cdots)^{p^\alpha})^{j_1}$$$$=(1-y_i^{p^\alpha}+y_i^{2p^\alpha}+\cdots)^{j_1}=1+j_1y_i^{p^\alpha}+A_i',~$$ 
where, $A_i'$ is a series in $K[[x,\sigma]]$ whose degree of the first term is greater than $2np^\alpha$. Hence, in all cases, $A_i^j$ can be written as $A_i^j=1+b_iy_i^{c_i}+A_i'$ where $c_i\in \N$, $b_i\in E$ and $A_i'$ is either $0$ or a series in $K[[x,\sigma]]$ whose degree of the first term is greater than $2nc_i$. Therefore, the relation 
	$A_{i_1}^{n_1}A_{i_2}^{n_2}\cdots A_{i_t}^{n_t}=1$ will be written in $K[[x,\sigma]]$ as follows: $$(1+b_{1}y_{i_1}^{c_1}+A_{i_1}')(1+b_{2}y_{i_2}^{c_2}+A_{i_2}')\cdots (1+b_{t}y_{i_t}^{c_t}+A_{i_t}')=1,$$ where $c_i>0$ and $b_i\in E^*$.  After expanding the left  side, we see that there is a unique product $b_1y_{i_1}^{c_1}b_2y_{i_2}^{c_2}\cdots b_ty_{i_t}^{c_t}=b_1b_2\cdots b_t y_{i_1}^{c_1}y_{i_2}^{c_2}\cdots y_{i_t}^{c_t}$, hence, this product is $0$ which is a contradiction since $\{y_0,y_1,\cdots,y_{n}\}$ generates a free $E$-algebra. Thus, $\{A_i=1+y_i\mid 0\le i\le n \}$ generates a free subgroup in $D=K(x,\sigma)$. As a corollary, $[{}_nx,A_0], [{}_nx,A_0^{-1}]\in\langle A_i\mid 0\le i\le n\rangle$ generates a free subgroup in $N$ by Lemma~\ref{2.2}.

	(ii) Assume that $a\in K\backslash E$ and $a$ is algebraic over $E$. Note that $\sigma (a)$ is also algebraic over $E$ because $E$ is the fixed field of $\sigma$. By induction, it is easy to see that $a^{(i)}=\sigma^i(a)$ is algebraic over $E$ for any $i\in \N$. In view of Part 1, we can assume that $E(a^{(i)}\mid i\in \N)$ is finitely generated. This implies that $K/E$ is a finite extension, so the Galois group $\Gal(K/E)$ is finite. If $d$ is the order of $\sigma$, then we have  $x^db=\sigma^d(b)x^d=bx^d$ for any $b\in K.$ Hence, $x^d$ belongs to the center $Z$ of $K(x,\sigma)$. It follows that every element of $K(x,\sigma)$ can be written in the form 
	$$\alpha_1+\alpha_2x+\cdots+\alpha_{d-1}x^{d-1},$$ with $\alpha_i\in Z$. This means that $K(x,\sigma)$ is a finite dimensional vector space over $Z$. By \cite[Theorem 2.1]{Pa_Go_84}, $N$ contains a non-abelian free subgroup.
\end{proof}	
\section{The existence of free subgroups in $D(x,y)^*$}

Let $D$ be a division ring with center $F$.  
Assume that $x,y$ are two commuting indeterminates and they also commute with all elements from $D$. We denote by $D(x,y)$ the division ring of fractions of the polynomial ring $D[x,y]=D[x][y]$ in $x,y$ over $D$; by
$R=D[x][[y]]$ the  power series ring in $y$ over the polynomial ring $D[x]$.
Let $N$ be a subnormal subgroup of $D(x,y)^*$ and assume that there exist non-commuting elements $a,b\in N\cap D$ such that $c=b^{-1}ab$ commutes with $a$.   
		
% The technique we use is from \cite{Pa_Go_Pa_15}, so we just summary results we need. 
% Firstly, $R$ is a left $R$-module with operation $\cdot$ which is an extension linearly over $D[[y]]$ of the rule $x\cdot d=da$ for every $d\in D$. 
Let $D[[y]]$ and $x$ act on $R$ by left multiplication and   via right multiplication by $a$ respectively, that is $f(y)\cdot \alpha=f(y)\alpha$ and $d\cdot x=xa$ for any $f(y)\in D[[y]], \alpha\in R, d\in D$. With this left action, $R$ is a left $R$-module. 
Assume that $K$ is a commutative subring of $R$ containing $a$. Then, it is shown in \cite{Pa_FeGoMa_05} that $R$ is an $(R,K)$-bimodule (left $R$-module and right $K$-module). If $M=1K+bK$ is the  right $K$-submodule of $R$ generated by $\{1,b\}$ and $S=\{ r\in R\mid r\cdot M\subseteq M \}$, then we can show that (or see \cite[Section 5]{Pa_Go_Pa_15}) $M$ is a free right $K$-module with the basis $\{1,b\}$ and $S$ is a subring of $R$ containing $K[[y]]$. This implies that the map $$\phi : S\to \End(M_K)\cong \M_2(K), r\mapsto \phi_r,$$ where $\phi_r: M\to M, \phi_r(m)=rm$, for every $m\in M$, is a ring homomorphism. Moreover, the image $\phi_a$ of $a$ via $\phi$ is the matrix  $\left[ {\begin{array}{cc}
 	a & 0 \\
 	0 & c \\
 	\end{array} } \right]$. Now consider four elements $$d_{11}=−b(ab-ba)^{-1} (bab^{-1}-x);$$ 
 $$d_{12}=(ab-ba)^{-1}(a-x);$$ $$d_{21}=-b^2(ab-ba)^{-1}(bab^{-1}-x);$$ and $$d_{22}=b(ab-ba)^{-1}(a-x).$$ Then, we have $d_{11}, d_{12},d_{21}, d_{22}\in S$ and 
 $$\phi_{d_{11}}=\left[ {\begin{array}{cc}
 	1 & 0 \\
 	0 & 0 \\
 	\end{array} } \right];\phi_{d_{12}}=\left[ {\begin{array}{cc}
 	0 & 1 \\
 	0 & 0 \\
 	\end{array} } \right];\phi_{d_{21}}=\left[ {\begin{array}{cc}
 	0 & 0 \\
 	1 & 0 \\
 	\end{array} } \right];\text{ and } \phi_{d_{22}}=\left[ {\begin{array}{cc}
 	0 & 0 \\
 	0 & 1 \\
 	\end{array} } \right].$$ We also conclude (or see \cite[Section 5]{Pa_Go_Pa_15}) $\phi(y)=\left[ {\begin{array}{cc}
 	y & 0 \\
 	0 & y \\
 	\end{array} } \right]$, hence, both $1+yd_{12}$ and $1+yd_{21}$ are invertible in $S$ and the images of $1+yd_{12}$ and $1+yd_{21}$ via $\phi$ are $\phi_{d_{12}}=\left[ {\begin{array}{cc}
 	1 & y \\
 	0 & 1 \\
 	\end{array} } \right]$ and $\phi_{d_{12}}=\left[ {\begin{array}{cc}
 	1 & 0 \\
 	y & 1 \\
 	\end{array} } \right]$ respectively.
 \begin{lemma}\label{3.2} For any $n\ge 1$, there exists $u\in S$ such that $u$ is invertible in $S$ and the image of $u$ via $\phi$ is $$\phi(u)=\left[ {\begin{array}{cc}
 		{1+(1-a^{-1}c)^{n+1}y} & {(1-a^{-1}c)^{n+1}y} \\
 		{-(1-a^{-1}c)^{n+1}y} & {1-(1-a^{-1}c)^{n+1}y} \\
 		\end{array} } \right].$$
 \end{lemma}
 \begin{proof} The element $u$ is constructed by modifying  the construction of $w$ in \cite[Proposition 5.8]{Pa_Go_Pa_15}. 
 	Let $K_0=F[a,a^{-1},c,c^{-1}]$ be the subring of $D$ generated by $\{a,a^{-1},c,c^{-1}\}$ and $K=K_0[[y]]$ be the subring of $D[[y]][x]$ generated by all series in indeterminate $y$ whose coefficients are in $K_0$.  We have $$\left[ {\begin{array}{cc}
 		{1+(1-a^{-1}c)^{n+1}y} & {(1-a^{-1}c)^{n+1}y} \\
 		{-(1-a^{-1}c)^{n+1}y} & {1-(1-a^{-1}c)^{n+1}y} \\
 		\end{array} } \right]=\I_2+ (a^{-1}c)^{n+1}y \left[ {\begin{array}{cc}
 		{1} & {1} \\
 		{-1} & {-1} \\
 		\end{array} } \right]$$$$=
 	\I_2+\phi(y)(1+a^{-1}c)^{n+1} (\phi(d_{11})+\phi(d_{12})-\phi(d_{21})-\phi(d_{22})).$$
 	According to \cite[Lemma 5.7]{Pa_Go_Pa_15},  for every $1\le i,j\le 2$ and for every $f\in F_0$, there exists an element $f'\in S\cap D[x]$ such that $\phi(f')=f\phi(d_{ij})$, which implies that there exists an inverse element $u\in S$ of $\I_2+\phi(y)(1+a^{-1}c)^{n+1} (\phi(d_{11})+\phi(d_{12})-\phi(d_{21})-\phi(d_{22}))$. Thus, $\phi(u)=\left[ {\begin{array}{cc}
 		{1+(1-a^{-1}c)^{n+1}y} & {(1-a^{-1}c)^{n+1}y} \\
 		{-(1-a^{-1}c)^{n+1}y} & {1-(1-a^{-1}c)^{n+1}y} \\
 		\end{array} } \right].$
 \end{proof}
 The following lemma is the key result in this section.
 \begin{proposition}\label{3.3} Let $D$ be a division ring, $D(x,y)$ the quotient division ring of the polynomial ring $D[x,y]$ in two commuting indeterminates $x,y$ and $N$ a subnormal subgroup of $D(x,y)^*$ with a series of subgroups $$N=N_n\triangleleft N_{n-1}\triangleleft \cdots \triangleleft N_0=D^*.$$ 
 Assume that  $a\in N\cap D$ and $b,c\in D$ with $ab=bc$, $a\ne c$ and $ac=ca$. Let $u$ be the element in Lemma~\ref{3.2}. Put $$v=[{}_na,1+y(ab-ba)^{-1}(a-x)] \text { and } w=[{}_na,1-yb^2(ab-ba)^{-1}(bab^{-1}-x)].$$ 
 		\begin{enumerate}
 			\item If $\Char(D)=0$, then $v,w$ generates a free subgroup in $N$.
 			\item If $\Char(D)>0$, then $[{}_nv,u]$ and $[{}_nw,u]$ generates a free subgroup in $N$.
 		\end{enumerate}
 		In particular, $N$ contains a non-abelian free subgroup.
 \end{proposition}
\begin{proof} Let $F$ be the center of $D$, $K_0=F[a,a^{-1},c^{-1}]$ the subring of $D$ generated by $\{a,a^{-1},c,c^{-1}\}$ and $K=K_0[[y]]$ $F_1=F_0[[y]]$ the subring of $D[[y]][x]$ generated by all series in indeterminate $y$ whose coefficients are in $K_0$. Let $d_{12}=(ab-ba)^{-1}(a-x)$ and $d_{21}=-b^2(ab-ba)^{-1}(bab^{-1}-x)$. Then, $1+yd_{12}, 1+yd_{21}$ is invertible in $S$ and 	$v=[{}_na,1+yd_{12}]$ and $w=[{}_na,1+yd_{21}]$.

	1. Assume that $\Char(D)=0$. By Lemma~\ref{3.1}, $v$ and $w$ belong to  $N$. We may show by inductive on $n$ that the image of $v$ and $w$ via $\phi$ are
	$\phi_{v}=\left[ {\begin{array}{cc}
		1 & (1-ac^{-1})^ny \\
		0 & 1 \\
		\end{array} } \right]$ and 		$\phi_{w}=\left[ {\begin{array}{cc}
		1 & 0 \\
		(1-ac^{-1})^ny & 1 \\
		\end{array} } \right]$ respectively. Observe that $(\phi_v-1)^2=(\phi_w-1)^2=0$ and  $(\phi_v-1)(\phi_w-1)=\left[ {\begin{array}{cc}
		(1-ac^{-1})^{2n}y^2 & 0 \\
		0 & 0 \\
		\end{array} } \right] $ is not algebraic over $\Q$, so that according to
	\cite[Lemma 5.3]{Pa_Go_Pa_15},  $\phi_v=\left[ {\begin{array}{cc}
		1 & (1-ac^{-1})^ny \\
		0 & 1 \\
		\end{array} } \right]$ and	$\phi_w=\left[ {\begin{array}{cc}
		1 & 0 \\
		(1-ac^{-1})^ny & 1 \\
		\end{array} } \right]$ generates a free subgroup in $\GL_2(K)$. Hence, $v$ are $w$, the inverse images of these matrices,  generates a free subgroup in $N$.
	
	2. Assume that $\Char(D)=p>0$. Put $\alpha = \left[\begin{array}{cc}
	0 & (1-ac^{-1})^n \\
	0 & 0 \\
	\end{array}\right]$ and $\beta=\left[\begin{array}{cc}
	0 & 0 \\
	1 & 0 \\
	\end{array}\right]$. Then, $\alpha^2=\beta^2=0$, and $\beta\alpha =\left[\begin{array}{cc}
	0 & 0 \\
	0&(1-ac^{-1})^n \\
	\end{array}\right]$ is not nilpotent. Hence, by \cite[Lemma 5.2]{Pa_Go_Pa_15}, $$\phi_v=1+y\alpha, \phi_w=1+y\beta\alpha\beta \text{ and } \phi_u=1+y(1-\beta)\alpha \beta\alpha (1+\beta)$$ generates the subgroup of $\GL_2(K)$ isomorphic to the free product $C_p*C_p*C_p$, where $C_p$ is the cyclic group of order $p$. Therefore, $\{u,v,w\}$, the set of inverse images of these matrices,  also generates the subgroup $\langle u,v,w\rangle$ of $S^*$ which is isomorphic to  the free product $C_p*C_p*C_p$. Observe that both $[{}_nv,u]$ and $[{}_nw,u]$ are in $N$ since $v,w\in N$ (Lemma~\ref{3.1} and Part 1). Applying Lemma~\ref{3.1} (3) and (4), we conclude that $[{}_nv,u]$ and $[{}_nw,u]$ generates a free subgroup in $N$.
\end{proof}

\section{A subnormal subgroup containing a solvable subgroup}

\begin{theorem}\label{4.1} Let $D$ be a division ring with center $F$, and assume that $N$ is a subnormal subgroup of $D^*$ containing a non-abelian solvable subgroup.  If transcendence degree of $F$ over the prime subfield $\mathbb{P}$ is infinite, then $N$ contains a non-abelian free subgroup.	
\end{theorem}
\begin{proof} Assume that $H$ is a non-abelian solvable subgroup of $N$, so $H^{(n)}=1$ for some $n>1$. Without loss of generality, we can assume $n=2$ because if $n>2$, we can consider the subgroup $H^{(n-2)}$ instead of $H$. Therefore, the derived subgroup $H'$ is abelian. Put 
	$${\mathcal S}=\Big\{ H'\leq A\leq H\mid A \text{ abelian} \Big\}.$$
	Let $G$ be a maximal element in ${\mathcal S}$ (there exists such a $G$ by Zorn's Lemma). Then, $G$ is an abelian normal subgroup of $H$. Take $a\in H\backslash G$ and $K=\mathbb{P}(G)$. Then, $K$ is a field and  $aKa^{-1}=K$. It follows that the subring  $K[a]$ of $D$ generated by $K\cup \{a\}$ can be written in the form $$K[a]=\Big\{\sum\limits_{i=0}^n a_0+a_1a+\cdots+a_na^n\mid n\in \N, a_0\in K \Big\}.$$ 
	Also, the map $\sigma : K\to K$ defined by  $\sigma(b)=aba^{-1}$ for every $b\in K$ is an automorphism of $K$. Consider the homomorphism
	$$\Phi\colon K[t,\sigma] \to K[a], \sum\limits_{i=0}^n a_it^i\mapsto \sum\limits_{i=0}^na_ia^i,$$
	where $K[t,\sigma]$ is the skew polynomial ring in an indeterminate $t$ over $K$ corresponding to $\sigma$. It is obvious that $\Phi$ is a surjection. If $ker(\Phi)\ne 0$, then $ker(\Phi)=\langle f(t)\rangle$ is generated by polynomial $f(t)\in K[t,\sigma]$. Therefore, $K[a]\cong K[t,\sigma]/\langle f(t)\rangle$ is a finite dimensional vector space over $K$. This implies that $K[a]$ is a division subring and by  \cite[Lemma 1]{Pa_AkEbKeGo_03}, it is centrally finite. Note that $N\cap (K[a])^*$ is a non-central subnormal subgroup of $(K[a])^*$ because $a\in N\cap (K[a])^*$. Hence, by \cite[Theorem 2.1]{Pa_Go_84}, $N\cap (K[a])^*$ contains a non-abelian free subgroup. Now, assume that $ker (\Phi)=0$. Then, $\Phi$ is isomorphism, so $\Phi$ can be extended naturally to the isomorphism $\overline \Phi : K(t,\sigma)\to K(a)$.  Let $K_1$ be the subfield of $K$ generated by  $\{\sigma^i(a)\mid i\in \Z\}$. Then, the restriction of $\sigma$ on $K_1$ is the isomorphism of $K_1$. Therefore, $K_1(t,\sigma)$ is a division subring of  $K(t,\sigma)$. Put $M_1=\overline \Phi(N)\cap K_1(t,\sigma)$. Then, $M_1$ is a subnormal subgroup of $K_1(t,\sigma)^*$ containing $t$.
	
	%		  B?y gi? ta l?p l?i ch?ng minh y h?t nh? trong ph?n sau c?a ch?ng minh \cite[Theorem 5.1]{Pa_Go_Pa_15} trong ?? thay th? \cite[Proposition 2.4]{Pa_Go_Pa_15} b?ng ??nh l? ~\ref{2.3} v? \cite[Propositon 5.8]{Pa_Go_Pa_15} b?ng M?nh ?? \ref{3.3}. C? th? nh? sau:
	
	If $K_1/P$ is infinitely generated, then by Theorem~\ref{2.3}, $M_1$ contains a non-abelian free subgroup. If $K_1/P$ is finitely generated, then  $\trdeg(K_1/P)$ is finite. But $\trdeg(F/P)=\infty$ implies the existence of elements $x,y\in F$ such that $x, y$ are algebraically independent over $\mathbb{P}$. By \cite[Lemma 7.1]{Pa_FeGoMa_05}, $x,y$ are also algebraically independent over $K(a)$. Then, $M=N\cap K(a)(x,y)^*$ is a non-central normal subgroup of  $K(a)(x,y)^*$ because $a\in M$. By Proposition \ref{3.3}, $M$ contains a non-abelian free subgroup. The proof of the theorem is now complete. 
\end{proof}

It is well known that every uncountable field has infinite transcendence degree over its prime subfield, so the proof of following corollary is elementary.
\begin{corollary}
	In a division ring $D$ with uncountable center, every subnormal subgroup of $D^*$ containing a non-abelian solvable subgroup contains a non-abelian free subgroup.
\end{corollary}

\begin{theorem}\label{4.2}
	Let $D$ be a division ring with center $F$. Assume that $N$ is a subnormal subgroup of $D^*$ containing a non-abelian solvable subgroup $H$. If $H$ is algebraic over $F$, then $N$ is contains a non-abelian free subgroup.
\end{theorem}
\begin{proof} Let $n$ be the derived length of $H$, that is, $n$ is the smallest positive  integer such that the $n$-th derived subgroup $H^{(n)}=1$. Since $H$ is non-abelian, $n\ge 2$. Put $G=H^{(n-2)}$. Then $G$ is a non-abelian metabelian subgroup. Set 
	$$\Re=\{ A\mid G'\le A\le G \text{ and } A \text{ is abelian }\}.$$ 
	By Zorn's Lemma, there exists a maximal element in $\Re$, say $M$. Then, $M$ is an abelian normal  subgroup of $G$. Put $K=F[M]$ and fix an element $a\in G\backslash M$. Clearly, $a$ normalizes $K$, that is, $aKa^{-1}=K$ and the set $L=\sum_{i\in \N} Ka^i$ is a subspace of the left vector space $D$ over a field $K$. By the fact that $a$ is algebraic over $F$, it is easy to see that $L\subseteq \sum_{i=0}^m Ka^i$ for some positive integer $m$.  Moreover, it is easy to that $aKa^{-1}=K$, which implies that $L$ is a division subring of $D$ and finite dimensional left vector space over $K$. By Zorn's Lemma, $K$ is contained in some maximal subfield $K_1$ of $L$, so $L$ is a finite dimensional left vector space over $K_1$, and in view of  \cite[(15.8)]{lam}, $L$ is centrally finite. Let $N_1=N\cap L$. Then $N_1$ is non-abelian subnormal in $L^*$. Therefore, $N_1$ contains a non-abelian free subgroup by \cite[Theorem 2.1]{Pa_Go_84}.
\end{proof}

\end{document}